\documentclass[journal, onecolumn]{IEEEtran}

\usepackage{bm}
\usepackage{enumerate}
\usepackage{mathrsfs,color}
\usepackage[numbers]{natbib}
\usepackage{amsfonts}
\usepackage{multirow}
\usepackage{amssymb}
\usepackage{caption}
\usepackage{comment}
\usepackage{blkarray}
\usepackage{amsmath}
\usepackage{float}
\usepackage{amssymb,amsthm, amsmath}
\usepackage{graphicx,afterpage}
\usepackage{epstopdf}
\usepackage{todonotes}
\usepackage[pdftex,colorlinks=true]{hyperref}
\usepackage[normalem]{ulem}
\usepackage{soul}
\usepackage{booktabs}
\usepackage{longtable}
\ifCLASSINFOpdf
\else
\fi
\newtheorem{theorem}{Theorem}[section]

\newtheorem{proposition}[theorem]{Proposition}

\newcommand{\R}{\mathbb{R}}

\newcommand{\bfx}{\mathbf{x}}

\newcommand{\bfU}{\mathbf{U}}

\newcommand{\bfy}{\mathbf{y}}

\newcommand{\E}{\mathbb{E}}
\newcommand{\bfw}{\mathbf{w}}
\newcommand{\calD}{\mathcal{D}}

\newcommand{\bfv}{\mathbf{v}}

\newcommand{\Prob}{\mathbb{P}}

\begin{document}
%
\title{Exploration Enhancement of  Nature-Inspired Swarm-based Optimization Algorithms}
%
%
%

\author{Kwok~Pui~Choi,
        Enzio~Hai~Hong~Kam,
        Tze~Leung~Lai,
        Xin~T.~Tong,
        and~Weng~Kee~Wong
  \thanks{K. Choi is with Department of Statistics and Applied Probability, National University of Singapore}
  \thanks{E. Kam is with the School of Computing, National University of Singapore }
  \thanks{T. L. Lai is with Department of Statistics, Stanford University}
  \thanks{X. T. Tong is with Department of Mathematics, National University of Singapore }
  \thanks{W. K. Wong is with Department of Biostatistics, University of California, Los Angeles }
  \thanks{Manuscript received ???, ????; revised ???, ????. This work is supported in part by Singapore MOE Academic Research Fund R-155-000-222-114, . }
}

%
%

\markboth{Journal of \LaTeX\ Class Files,~Vol.~14, No.~8, August~2015}%
{Shell \MakeLowercase{\textit{et al.}}: Bare Demo of IEEEtran.cls for IEEE Journals}
%



\maketitle

\begin{abstract}
Nature-inspired swarm-based algorithms  have been widely applied to tackle  high-dimensional and complex optimization problems across many disciplines.  They are general purpose optimization algorithms, easy to use and implement, flexible and assumption-free. A common drawback of these algorithms is premature convergence and the solution found is not a global optimum.
We  provide  sufficient conditions for an algorithm to converge almost surely (a.s.) to a global optimum. We then propose a general, simple and effective strategy, called Perturbation-Projection (PP), to enhance an algorithm's  exploration capability so that our convergence conditions are guaranteed to hold.
We illustrate this approach using three widely used nature-inspired swarm-based optimization algorithms: particle swarm optimization (PSO), bat algorithm (BAT) and competitive swarm optimizer (CSO). Extensive numerical experiments show that each of the three algorithms with the enhanced PP strategy outperforms the original version in a number of notable ways.
\end{abstract}


%
\IEEEpeerreviewmaketitle

%
%
%
%


\section{Introduction}

\IEEEPARstart{O}{ver}  the past decades, Swarm Intelligence has and continues to inspire a steadily rising numbers of nature-inspired  swarm-based algorithms  in optimizing  high-dimensional complex cost functions, including those that do not have  analytic forms. Examples of swarm-based algorithms include  particle swarm optimization (PSO), competitive swarm optimizer (CSO), ant colony optimization, bat algorithm (BAT) etc.
Swarm-based optimization algorithms are mainly motivated by nature or animal behavior and then thoughtfully formulated into an algorithm that iterates to the optimum according to a couple of equations.   Generally, these algorithms are  easy to code and implement, and do not need gradient information or other sophisticated mathematical techniques to work well. Computer codes are  widely and freely available, which have undoubtedly helped fuel numerous and various applications of these algorithms to tackle many different types of complex real-world optimization problems.  Documentation of their effectiveness is widespread resulting in  their meteoric applications in an increasing number of sectors in industry and  in academia \cite{whit1,whit2}.

There are  known shortcomings of  nature-inspired swarm-based optimization algorithms. We highlight two: (1) They require effective tuning of their parameters to achieve optimal performance \citep{HLY2019} and  (2) they  often suffer from premature convergence to a local optimum of the cost function, especially  when the cost function is high-dimensional and multi-modal. Recently, \cite{LCTW2020} proposed a general strategy to tackle the first shortcoming using statistical tools.

In this paper, we propose an innovative and simple strategy to address the latter problem by modifying a nature-inspired swarm-based algorithm  to ensure its  almost sure (a.s) convergence to a global optimum.  The proposed methodology is general, as long as certain unrestrictive technical conditions are met, and we demonstrate the techniques using three popular swarm-based algorithms mentioned above. Theorem \ref{ascon}  in Section II shows that if a swarm-based algorithm satisfy conditions (C1)--(C3), the  algorithm is  guaranteed to converge a.s. to a global solution.  We also provide comments on possible extension of the proposed methodology to evolutionary algorithms, like differential evolution.

We next  propose  a  general, yet simple and  effective strategy, Perturbation-Projection (PP), to modify an algorithm, if necessary, to ensure  conditions (C1) and (C2) to hold. In Section III, we demonstrate how the PP strategy is applied to the original PSO, BAT and CSO algorithms with appropriate modifications to meet the required conditions for almost sure convergence to the global optimum. We  denote these modified  algorithms of their counterparts by mPSO, mBAT and mCSO.
When it is necessary to further improve the exploitation ability of an  algorithm but not at the  expense of its exploration ability, we introduce  a Heterogeneous Perturbation-Projection (HPP) strategy in Section IV to satisfy conditions (C1)--(C2).
In Section V, we  conduct extensive numerical experiments  using many  commonly used test functions and show that the modified algorithms remain very competitive with the original algorithms.
We also conduct tuning parameter analysis for the PP and HPP strategies. The paper ends with conclusion and related future work.

 In the literature, there are many different modifications of the original algorithms or new techniques   to address the premature convergence issue for swarm-based algorithms; see, for example \cite{parsopoulos2001objective,stacey2003particle, krohling2005gaussian, elshamy2007clubs}.  These methods tend to either  apply to only one specific algorithm, e.g. PSO, instead of a class of algorithms, or, require additional  detection or tuning mechanisms that can be complicated to implement.  A commonality among these works is that they do not have rigorous mathematical theory to support the modifications for improved convergence to the global optimum. The distinctive features of our proposed modifications are that they are supported by mathematical theory, simple to implement and applicable  to  any swarm-based algorithm.

\section{Stochastic enhancement of an algorithm's  exploration}

In this section, we introduce  a simple and  effective strategy  to modify  a  nature-inspired swarm-based  algorithm to ensure  the modified algorithm  converges almost surely (a.s.) to a global optimum.  The cost function can  be a high-dimensional, non-differentiable, non-separable or non-convex with multiple local or global optima. The key idea behind the proposed algorithms is simple; we incorporate an additional  stochastic component with an appropriate  noise level to a nature-inspired swarm-based algorithm to enable it to escape from a local minimum. While this idea may not be entirely new,  the distinguishing feature of proposed modification is its generality and the modified algorithm is guaranteed to  converge a.s.  to a global optimum.

Without  loss of generality, we assume that we have a minimization problem as $\min_{\bfx \in \calD} f(\bfx) = - \max_{\bfx \in \calD}\{- f(\bfx)\}$. Here $f(\bfx)$ is a user-specified real-valued cost function defined  on    a given compact subset $\calD$ of $\R^d$.
For many real-world applications,  there are physical and budgetary constraints and we assume that these constraints  have been appropriately translated into  the compact search space. For many other problems, like those  in machine learning, regularization terms, sometimes denoted by $\rho\|\bfx\|_1$ (or $\rho\|\bfx\|^2_2$)  are  often added to the nonnegative loss/cost function $f$ to prevent over-fitting,  where the regularization parameter $\rho$ can often be selected through various cross validation methods  \cite{james2013introduction,shalev2014understanding}.
In  such situations,  we  minimize the function $F(\bfx)=f(\bfx)+\rho\|\bfx\|_1$ (or $f(\bfx)+\rho\|\bfx\|^2_2$) over $\calD$, which is the closed $L_1$-ball (or $L_2$-ball)  centered at the origin with  radius  $  f(\mathbf{0})/\rho$. This is because for $\bfx\notin\calD$, $F(\bfx)>F(\mathbf{0})$ so the global minima are in $\calD$.

To solve the optimization problem,  we resort to nature-inspired swarm-based algorithms, which are increasingly used to solve complex and high-dimensional optimization problems.  They generally employ an evolutionary-like or swarm-based strategy to search for an optimum by first randomly generating a user-specified set of candidate solutions. Depending on the particular swarm-based algorithm, candidate solutions are called agents or particles. These algorithms embrace common principles that include good exploration and exploitation abilities and they have stochastic components and tuning parameters.  For example, for a swarm-based type of algorithm, it can be typically described as follows.

At time or iteration $t=1, 2, \ldots $,  let $\bfx_i(t)$ represent the $i$-th particle's position or candidate solution of the optimization problem. Let $\{\bfx_i(t): 1 \le i \le n\}$ denote the collection of candidate solutions of the swarm of size $n$.  In addition, there is auxiliary information  from each member of the swarm and we denote them collectively by $\{\bfv_i(t): 1 \le i \le n\}$.  For example, in PSO, the auxiliary information of each particle is its velocity with which it flies to the current position.
We denote them by
$$\bfx(t)=(\bfx_1(t),\ldots, \bfx_n(t))~\text{and}~  \bfv(t)=(\bfv_1(t),\ldots, \bfv_n(t)).$$
Suppose the algorithm has  a stochastic update rule of the following form: for agent $i$
\begin{eqnarray*}
 \bfx_i(t+1)&=&\Phi_i(\bfx(t),\bfv(t),f), \\
 \bfv_i(t+1)&=&\Psi_i(\bfx(t),\bfv(t),f),
\end{eqnarray*}
for $ 1 \le i \le n$, where $\Phi_i, \Psi_i $ are some functions of which the outputs can be random.
At either a pre-selected deterministic time or a stopping time $T$,  the algorithm is terminated and the best solution  $\min_{1 \le i \le n} f(\bfx_i(T))$ is output.

A key desirable  property of a global optimization algorithm is its ability to identify a global optimum eventually.
To achieve convergence to the global optimum, \citet{TZ1989}   remark  that  the algorithm has to  be able to explore densely in the search space.
Inspired by   their    remark, 
 we formulate  the following conditions on the algorithm's agents,  $\bfx_1(t),\ldots, \bfx_n(t)$.
\begin{itemize}
\item[{\rm (C1)}] At any iteration, all agents  stay in  $\calD$, i.e., $\bfx_i(t) \in \calD$ for $1 \le i \le n$ and $t= 0, 1, \ldots, T$.
\item[{\rm (C2)}] Conditional on the information up to  the current iteration  $t$, there exists an  agent at the next iteration, $t+1$, that has a positive probability to explore a region of improvement in $\calD$. Specifically,
there exists an $\alpha >0$, independent of $t$, such that for any $d$-dimensional ball $B$ in $\calD$ satisfying
\[
f(\bfy) \le \min_{1\le i\le n} f(\bfx_i(t)),\quad \mbox{ for  all } \bfy\in B,
\]
 the following holds
$$ \max_{1 \le i \le n} \Prob_t( \bfx_i(t+1) \in B ) \ge \alpha |B|, $$
with $|B|$ denoting the  volume of $B$ and $\Prob_t$  the conditional probability with information at the $t$-th iteration.

\item[{\rm (C3)}] The algorithm is time improving, i.e.,
\[
\min_{1\le i \le n} f(\bfx_i(t+1))\leq \min_{1 \le i \le n} f(\bfx_i(t)), \quad a.s.
\]
\end{itemize}

\begin{theorem}
\label{ascon}
Suppose $f$ is a continuous function defined on a given compact subset $\calD$ of $\R^d$ with at least one minimizer in the interior of  $\calD$.
Any swarm-based algorithm that satisfies {\em (C1)}--{\em (C3)} will  converge to the global minimum almost surely (a.s). In other words, if the swarm-based algorithm has  $n$ agents,
$
\lim_{t\to \infty} \min_{1 \le i \le n}f(\bfx_i(t)) = \min_{\bfx\in \calD} f(\bfx)$  a.s..
\end{theorem}
\begin{proof}
Let an interior point $\bfx_* \in \calD$ denote a minimizer of $f$.
Without loss of generality, we assume $f(\bfx_*)=0$.
By (C3),  $m(t):=\min_{1 \le i \le n} f(\bfx_i(t))$ is a non-increasing sequence, thus it suffices  to show that $m(t) \stackrel{P}{\to} 0 $ as $t \to \infty$.
 Given a fixed $\epsilon >0$, by continuity of $f$, there exists  $\delta>0$ such that
 $0 \le f(\bfx) \leq \epsilon $ for $\bfx\in B(\bfx_*, \delta):=\{\bfx:\|\bfx-\bfx_*\|\leq \delta\}$.  Since $\bfx_*$ is an interior point, $B(\bfx_*, \delta) \subset \calD$ for a sufficiently small $\delta$. By (C2), at each iteration, there is an agent $\bfx_i$ such that either $f(\bfx_i(t))\leq \epsilon$ or
\[\Prob_t\left(\bfx_{i(t+1)}(t+1)\in B(\bfx_*, \delta) \right)\ge 
\alpha |B(\bfx_*, \delta) | =: \alpha(\epsilon).
\]
If $f(\bfx_i(t))\leq \epsilon$, then $m(t)\leq \epsilon$.
If $\bfx_{i(t+1)}(t+1)\in B(\bfx^*, \delta)$, then $
m(t+1)\leq f(\bfx_{i(t+1)}(t+1))\leq \epsilon$.
By the tower property of conditional expectation,
\begin{align*}
& \Prob(m(t)> \epsilon) \\
&\leq \Prob(\bfx_{i(s+1)}(s+1) \notin B(\bfx^*, \delta) , \text{ for } s=1,\ldots, t-1)\\
&=\E \prod_{s=1}^{t-1}\Prob_s(\bfx_{i(s+1)}(s+1)\notin B(\bfx^*, \delta) ) \\
&< \E \prod_{s=1}^{t-1}(1-\alpha(\epsilon))=(1-\alpha (\epsilon))^{t-1}.
\end{align*}
Therefore, $m(t) \stackrel{P}{\to} 0$ as $t\to \infty$.
\end{proof}

Theorem \ref{ascon}  shows that any swarm-based algorithm satisfying  (C1)--(C3) is guaranteed to converge to a global optimal solution eventually. Many swarm-based algorithms satisfy (C3), but not all swarm-based algorithms stay within the search space $\calD$ in their iterations thus violating condition (C1) or they are exploratory in the sense of condition (C2). However,  Section II-A below shows that, with slight modifications of an algorithm, we can ensure the  modified algorithm satisfies (C1)--(C2).

\subsection{Perturbation-projection strategy}

A first step is to ensure all agents  stay in $\calD$. For this purpose, we  introduce  the projection onto $\calD$  ensuring the  agents stay inside the search space, $\calD$, a compact subset in $\R^d$:
\[
\chi_\calD(\bfx)=\arg\min_{\bfy\in \calD} \|\bfy-\bfx\|_2,
\]
where $\|\bfy-\bfx\|_2$ denotes the Euclidean distance between  $\bfx$ and $\bfy$. In other words, $\chi_\calD(\bfx)$ is a point in $\calD$ that is closest to $\bfx$. When  ties occur, they can be broken arbitrarily. By definition,  $\chi_\calD(\bfx)=\bfx$ if $\bfx\in \calD$. In general, adding the projection to an algorithm's iterations reduces  the agents' distances to the optimal solution \cite{shalev2014understanding}. In particular,  if $\bfx_*\in \calD$ and $\calD$ is convex, it can be shown that
$
\|\chi_\calD(\bfx)-\bfx_*\|\leq \|\bfx-\bfx_*\|_2.
$ In many optimization problems, the search space, $\calD$, is usually a $d$-dimensional hyper-rectangle, it follows that $\calD$  is convex.

A second step  is to enhance  an algorithm's exploration. We achieve this by injecting noise into its  dynamics. In stochastic optimization algorithms, such as the perturbed gradient descent and the Langevin algorithm,  Gaussian noise is added to the classical gradient descent algorithm. This  injection of noise  effectively helps the algorithm to get out of saddle points or local minimums efficiently \cite{ge2015escaping,chen2020stationary}. We  adopt the same strategy  in this article due to its simplicity. Other form of  noise injection also exists. For example, in multi-arm bandit problem, one can enforce exploration in the classical greedy algorithm by allocating  a small proportion  of trials on random arms.

By combining these two steps, the modified algorithm can be formulated as
\begin{equation}
\label{eqn:pp}
\bfx_{i}(t+1)=\chi_\calD (\chi_\calD(\bfx'_i(t+1))+\bfw_i(t)),
\end{equation}
where $\bfx'_i(t+1)$ is the position of the $i$th agent after applying the algorithm's update rule at $t+1$, and $\bfw_i(t)$ is an independent draw from a density $p_t$ such that there is a constant $\alpha>0$
\begin{equation}
\label{eqn:psupport}
p_t(\bfy-\bfx)>\alpha,\quad \mbox{for } \bfx,\bfy\in \calD.
\end{equation}
The noise density $p_t$ in general can be adaptive,  depending on the evolution of the algorithm up to time $t$. Here we only need it to be strictly positive and  bounded away from zero.
We call this method of modification in \eqref{eqn:pp} a perturbation-projection (PP) modification or strategy. The proposition below shows that with PP strategy, the modified algorithm satisfies (C1) and (C2).

\begin{proposition}
\label{prop:mod}
Suppose $A$ is  a swarm-based algorithm and at each iteration, every agent is projected onto $\calD$. If
 the perturbation-projection modification \eqref{eqn:pp} is applied to at least one of the agents, then {\em (C1)} and {\rm (C2)}  hold.
\end{proposition}
\begin{proof}
Since $\chi_{\calD}$ is applied to all agents, (C1) clearly holds.
At the $(t+1)$th iteration,  denote the agent to which  the modification is applied  by  $i=i(t+1)$. Let  $\bfy=\chi_\calD(\bfx'_{i}(t+1))$.

Then, for any  $d$-dimensional ball $B$  in $\calD$,
\begin{eqnarray*}
\lefteqn{ \max_{1 \le k \le n} \Prob_t( \bfx_k(t+1) \in B ) } \\
&\geq& \Prob_t(\bfx_{i(t+1)}(t+1)\in B)\\
&\geq& \Prob_t(\bfy+\bfw_i(t)\in B)\\
&=&\int_B \Prob_t(\bfy+ \bfw_i(t)\in d\bfx)\\
&=&\int_B \Prob_t(\bfw_i(t) \in d(\bfx-\bfy))\\
&=&\int_B p_t(\bfy-\bfx)d\bfx \ge \alpha |B|,
\end{eqnarray*}
showing (C2) holds.
\end{proof}

\section{Applications}
In this section, we illustrate   how to implement our PP strategy to three  nature-inspired swarm-based optimization algorithms, namely, particle swarm optimization (PSO),  bat algorithm (BAT) and competitive swarm optimization (CSO).

\subsection{PSO and mPSO}
 \citet{KE1995}  proposed the particle swarm optimization (PSO) algorithm. The algorithm has stimulated  many refinements and inspired many variants, and these algorithms find  wide applications in many fields. A recent search of ``particle swarm optimization'' in the Web of Science generated  more than to 32,000 articles and almost 600 review articles. A small sampling of early examples of how PSO has been studied and modified  over the years are \cite{shi98,BK2007,poli,ES2000,CK2002,van2000cooperative,CK2002,T2003,CL07,PC2010}.  PSO probably has the most modified versions among nature-inspired metaheuristic algorithms and  they continue to this day.

 PSO algorithm models after  the movement  of a flock  of birds looking for food collectively. To align with our terminology in this article, we think of birds in PSO as agents, and food as a global minimizer of an objective function. At the next iteration, $t+1$, each agent veers towards the best location it has found (cognitive/memory component) and the best location known to the flock up to iteration $t$ (social component). To describe the memory effect, we allocate a personal ``memory" agent $\bfx^*_i(t)$ to record the best location agent $i$ has found up to iteration $t$, and a global ``memory" agent $\bfx^*(t)$ to record the best location found by all agents collectively so far.
Recall the basic  PSO algorithm runs the following iterations iteratively:
\begin{enumerate}
\item Generate two independent random vector $\bfU_1$ and $\bfU_2$  from the uniform distribution on $[0,1]^d$ and let
\begin{eqnarray*}
  \bfv_i(t+1)&=&w\bfv_i(t)+c_1\bfU_{1}\circ  (\bfx_i^*(t)-\bfx_i(t)) \\
&& \quad + c_2 \bfU_{2}\circ   (\bfx^*(t)-\bfx_i(t)),
\end{eqnarray*}
where $w,c_1, c_2$ are given constants, and $\circ$ denotes the Hadamard product.
\item Update $\bfx_i(t+1)=\bfx_i(t)+\bfv_i(t+1)$.
\item Update personal best
\[
\bfx^*_i(t+1)=\begin{cases}
\bfx_i(t+1), &\mbox{if }f(\bfx^*_i(t))>f(\bfx_{i}(t+1));\\
\bfx^*_i(t), &\mbox{otherwise}.
\end{cases}
\]
\item Repeat steps 1--3 for all agents, then  update the global best agent
\[
\bfx^*(t+1)=\begin{cases}
\bfx_j^*(t+1), & \mbox{if condition H holds}, \\
\bfx^*(t), &\mbox{otherwise},
\end{cases}
\]
where $ j=\text{argmin}_{1\leq i\leq n} f(\bfx^*_i(t+1))$, $m(t)= \min_{1\le i \le n} f(\bfx^*_i(t))$ and condition H:  $m(t+1) < m(t)$.
\end{enumerate}
Our modified PSO, denoted by mPSO, is to apply the PP strategy in step 2:

2') Update
\begin{equation}
\label{eqn:csomod}
\bfx_i(t+1)=\chi_{\calD}\left(\chi_{\calD}\left(\bfx_i(t)+\bfv_i(t+1)\right)+\bfw_i(t)\right),
\end{equation}
where $\bfw_i$'s are independent multivariate normal distributed with mean vector $\mathbf{0}$ and covariance matrix $\sigma \mathbf{I}$.

Since its introduction in 1995, there have been  numerous attempts to analyze the basic PSO convergence behavior. In our opinion, \citet{YY2015} provide the first rigorous proof of the weak convergence of PSO to a global optimum, without overly restrictive assumptions. Moreover, using stochastic approximation technique, they provide the rate of convergence. In another direction, \citet{TCLW2020} introduce two smoothed versions of PSO and prove that they converge a.s. to a cost function's global optimum.

The  next proposition shows that if we modify the basic PSO algorithm by replacing the update step 2 by the new noise enhanced update step 2', the modified mPSO algorithm  converges to a global minimum of the cost function a.s.
\begin{proposition}
\label{prop:PSO}
 Let $f$ be a continuous function defined on a compact subset  $\calD$ of $\R^d$ and one of its minimizers is in the interior of $\calD$.
 Then the algorithm mPSO  converges to $\min_{\bfx \in \calD} f(\bfx)$ a.s..
  \end{proposition}
\begin{proof}
%

By Theorem \ref{ascon}, it suffices to  verify that mPSO satisfies (C1)--(C3). Note that $\bfx_i(t+1) \in \calD$ by projection. As $\bfx_i^*(t+1)$ takes either value $\bfx_i^*(t)$ or $\bfx_i(t+1)$, therefore  $\bfx_i^*(t+1) \in \calD$ if $\bfx_i^*(t)\in \calD$.
Similar argument applies to $\bfx^*(t+1)$,  so (C1) holds  by induction.
Since we add noise to all agents except the  $\bfx_i^* (1 \le i \le n) $ and $\bfx^*$ agents, Proposition \ref{prop:mod} implies (C2) holds. Consider the global memory agent, by step 4, $f(\bfx^*(t+1)) \le f(\bfx^*(t))$, so (C3) is verified. This completes the proof. 
\end{proof}

\subsection{BAT and mBAT}
The success of bats, dolphins, shrews  and other animals in applying echolocation technique to hunt and navigate inspired
 \citet{yang2010} to propose the Bat Algorithm (BAT) in 2010. BAT meets with rising popularity in applications. According to a recent search of the Web of Science,  over 1,500 articles, nearly 800 proceedings papers and 48 review papers have been written on this algorithm and its variants  in just ten years. We refer interested readers in BAT and its wide ranging applications to the  review articles in  \cite{AAAr2012, YH2013} and the references therein. Recent applications of the BAT algorithm  include \cite{cai,xue}. As with other metaheuristic algorithms, BAT has also many modified or hybridized versions for improved performance in various ways; some examples of the modified versions can be found in  \cite{osaba,binu,wangchu,khooban,luxiawang} and some examples of BAT algorithm hybridized with another metaheuristic algorithm are \cite{heding, shrichandran2017,kishore}.

Let $\bfx_i(t)$, $1 \le i \le n$, denotes the position of the $i$-th bat at $t$-th iteration, and $\bfx^*(t)$  tracking the best position found by the cauldron of $n$ bats.   Recall   BAT runs the following iterations iteratively:
\begin{enumerate}
\item Generate $n$ independent random number $U_i$ from the uniform distribution on $[f_{\min},f_{\max}]$,
\begin{eqnarray*}
  \bfv_i(t+1)&=& \bfv_i(t)+U_{i} (\bfx_i(t)-\bfx^*(t)).
\end{eqnarray*}
\item  Let $r_0$ (pulse rate) denote  a user-specified threshold.
Generate $n$ independent random numbers $r_i$ from the uniform distribution on $[0,1]$, $1\le i \le n$.

Update according to  (a): $r_i < r_0$ or (b)$r_i \ge r_0$:

(a) If $r_i < r_0$,   $\bfx_i(t+1)=\bfx_i(t)+\bfv_i(t+1)$;

(b) If $r_i \ge r_0$,
$\bfx_i(t+1)=\bfx^*(t)+\epsilon_i(t)$, where the components of $\epsilon_i(t)$ are independent mean 0 normal distribution with standard deviation 0.001.

\item Generate $n$ independent random numbers $r_i$ from the uniform distribution on $[0,1]$. If $r_i$ is less than a threshold $r_A$ (loudness), or
$f(\chi_{\calD}(\bfx_i(t)))<f(\bfx_i(t+1))$ update $\bfx_i(t+1)=\chi_{\calD}(\bfx_i(t))$.
\item Update global best
\[
\bfx^*(t+1)=\bfx_{i^*}(t+1),
\]
where $
i^*=\arg\min_i f(\bfx_i(t+1)).
$
\end{enumerate}
In the procedure above, $f_{\min}, f_{\max}$ are given constants which describe the minimal and maximal frequencies.
The threshold probabilities $r_0$ and $r_A$ describe the pulse and emission rates.  The random noise $\epsilon_i(t)$ describes the loudness of each bat.
\citet{yang2010} also suggests using time varying $\epsilon_i(t)$ and $r_0$ for improved performance. Here we fix them as in the standard values suggested by the MATLAB package provided by the same author.

Our modified BAT, denoted by mBAT, is to apply the PP strategy by an additional step after  step  3

3') Update
\begin{equation}
\label{eqn:batmod}
\bfx_i(t+1)=\chi_{\calD}\left(\chi_{\calD}\left(\bfx_i(t)+\bfv_i(t+1)\right)+\bfw_i(t+1)\right).
\end{equation}

\begin{proposition}
\label{prop:BAT}
 Let $f$ be a continuous function defined on   $\calD$, a compact subset of $\R^d$. Suppose further that a minimizer of $f$ lies in the interior of $\calD$, then
  the algorithm mBAT converges to $\min_{\bfx \in \calD} f(\bfx)$ a.s..
  \end{proposition}
\begin{proof}
%

By Theorem \ref{ascon}, it suffices to  verify mBAT satisfies (C1)--(C3). Note that $\bfx_i(t+1) \in \calD$ by projection. As $\bfx^*(t+1)$ takes either value $\bfx^*(t)$ or $\bfx_i(t+1)$ for some $i$, therefore it is in $\calD$ if $\bfx_i^*(t)\in \calD$. So (C1) holds  by induction.

To verify (C2), we denote the value of $\bfx_i(t+1)$ after step 3') \eqref{eqn:batmod}  as $\bfy_i(t)$. Following the proof of Proposition \ref{prop:mod}, we can show that for any d-ball $B\subset \calD$,
\[
\Prob(\bfy_i(t)\in B)\geq \alpha |B|.
\]
Then if all $\bfy\in B$ satisfies $f(\bfy)<f(\bfx_i(t))$, we find that
\[
\Prob(\bfx_i(t+1)\in B)\geq (1-r_A)\Prob(\bfy_i(t)\in B)\geq \alpha (1-r_A) |B|.
\]
So (C2) holds.

By step 4, $f(\bfx^*(t+1)) \le f(\bfx^*(t))$, so (C3) is verified. This completes the proof of Proposition \ref{prop:BAT}.
\end{proof}

\subsection{CSO and mCSO}
Inspired by the particle swarm optimization, \citet{CJ2014}  in 2014 proposed the Competitive Swarm Optimizer (CSO)  for
``large scale optimization problems and is able to effectively solve problems of dimensionality up to 5000''.  It is generally deemed to outperform PSO and it is also continually improved in various ways such as  \cite{mohapatra, bowang}.  Most recently, \cite{weili}  applied CSO to optimize a multitasking problem and \citet{zhang2020cso}  modified CSO to find optimal experiment designs to estimate interesting parameters in a nonlinear regression model with several interacting factors.

We shall call the particles in CSO   agents.  In each CSO iteration, agents are randomly paired and their functional values are compared.
In each pair, agent with smaller functional value is declared  winner and its configuration remains unchanged in the next iteration. The other agent in the pair  is declared  loser and his configuration in the next iteration will move closer to the winner's configuration and the average configuration of all the agents, see steps 4 and 5 below. Specifically,  CSO  repeats steps 1 to 5 until termination criterion is met.
\begin{enumerate}
\item Randomly pair the agents. Repeat the  steps 2 to 5 below for each pair.
\item Consider a pair of agents indexed as  $i$ and $j$. Define  $w=i,\ell=j$ if $f(\bfx_{i}(t))<f(\bfx_j(t))$. Otherwise $w=j,\ell=i$.
\item Let $\bfx_w(t+1)=\bfx_w(t)$, $\bfv_{w}(t+1)=\bfv_w(t)$.
\item Generate three independent random vectors $U_1,U_2,U_3$  from the uniform distribution on $[0,1]^d$ and let
\begin{eqnarray*}
\bfv_{\ell}(t+1)&=&U_{1}\circ \bfv_{\ell}(t)+ U_2\circ(\bfx_w(t)-\bfx_{\ell}(t)) \\
&& \quad +\phi U_3\circ(\bar{\bfx}(t)-\bfx_l(t)),
\end{eqnarray*}
where $\circ$ denotes element-wise product, $\phi$ a given constant, and $\bar{\bfx}(t)=\sum_{i=1}^n \bfx_i(t)/n$.
\item Update $\bfx_{\ell}(t+1)=\bfx_{\ell}(t)+\bfv_{\ell}(t+1)$.
\end{enumerate}

Our modified CSO, denoted by mCSO, is to apply the PP strategy in  step 5 only to the loser agent:

5') Update
\begin{equation}
\label{eqn:csomod}
\bfx_{\ell}(t+1)=\chi_{\calD}\left(\chi_{\calD}\left(\bfx_{\ell}(t)+\bfv_{\ell}(t+1)\right)+\bfw_{\ell}(t)\right).
\end{equation}
It is not necessary to apply projection on $\bfx_w(t+1)$ because $\bfx_w(t+1)=\bfx_w(t) \in \calD$.

\begin{proposition}
\label{prop:CSO}  Let $f$ be a continuous function over   $\calD$, a compact subset of $\R^d$. Suppose further that a minimizer of $f$ lies in the interior of $\calD$, then 
  the algorithm mCSO converges to $\min_{\bfx \in \calD} f(\bfx)$ a.s..
\end{proposition}
\begin{proof}
%
%
%
By Theorem \ref{ascon} and Proposition \ref{prop:mod}, our claim  will be proved if we can show that mCSO satisfies (C1)--(C3).
Since  $\bfx_{\ell}(t+1) \in \calD$ by projection, and  $\bfx_w(t+1)=\bfx_w(t)$, (C1)  holds. 
As we add noise to  the loser agents, (C2) holds by Proposition \ref{prop:mod}. 
Let  $i_*=\arg \min_{1 \le i \le n} f(\bfx_i(t))$.
In the next iteration, $\bfx_{i_*}(t)$ will be a winner no matter  whom he is paired with. So $\bfx_{i_*}(t+1)=\bfx_{i_*}(t)$; and
$
\min_i f(\bfx_{i}(t+1))\leq
 f(\bfx_{i_*}(t+1))= \min_i f(\bfx_{i}(t)).
$
This proves that (C3) is satisfied and hence  the proof of the proposition.
\end{proof}

\section{Heterogeneous PP}
In applications, the strength of stochastic perturbation in the PP strategy often needs proper tuning.
In general, a stronger perturbation enhances  exploration at the expense of exploitation. During the final stage of optimization, the agents will have difficulties in narrowing down the optimal solution with strong perturbations. There can be different ways to solve this issue such as  choosing a diminishing  perturbation as in the original BAT algorithm; using adaptive adjustment mechanism; or using reinforcement learning tools to find the optimal parameter. However,  these approaches lead to further tuning parameters or additional computational complexity. Instead,  we propose a simpler strategy, which is nature-inspired,  that involves role heterogeneity of members in  the swarm.

For swarms in nature, it is common to observe they adoption division of labors. Agents specializing in different tasks collaborate  to improve  collective effectiveness. In the same vein, we set some  agents in the swarm to  specialize in exploration while others  in exploitation. In the context of PP strategy, we only perturb the  exploration agents so they help exploring the search  space while we do not perturb  exploitation agents so as not to hamper their convergence capability. We call this heterogeneous perturbation-projection, abbreviated to HPP, strategy.
In contrast with the other ways to find optimal perturbation strength, HPP strategy is very easy to implement. For example, if we want to have half exploration agents and half exploitation agents, we will apply PP strategy only to  the first half of the agents while the other agents we only apply the projection modification.  This idea can be applied directly to mPSO and mBAT. As for mCSO, because pairs are randomly formed,  we will apply PP strategy only to  the loser agents   in the first half of pairs.

Our HPP strategy can be considered as a special case of cooperative or collaborative learning, which suggests that mixing agents of different settings can lead to improved overall performance. Existing works on cooperative swarm-based algorithms can be found in \cite{van2000cooperative, van2004cooperative, zhang2019cooperative}. In comparison, our HPP strategy focuses more on the cooperation between exploration agents and exploitation agents.
Recently, \cite{dong2020replica} also applies this idea to local optimization algorithms such as gradient descent and Langevin algorithm.

Note  that Proposition \ref{prop:mod}
only requires the perturbation strategy applied to at least  one agent, so this proposition is still applicable  to algorithms with HPP, which perturbs half of their agents (or loser agents in CSO). Consequently,  Theorem \ref{ascon} guarantees H-PP modified algorithms  converge a.s to a global optimal solution. We denote the modified version of an algorithm $A$ using H-PP strategy by $hmA$.

\section{Numerical experiments}

In this section, we conduct extensive numerical experiments to investigate  whether our modified versions of $A$ ($A =$ PSO, BAT or CSO), $mA$ and $hmA$, outperform or at least perform on par with $A$  given a computing budget for a large class of very different types of cost functions.  We use two  measures to evaluate the performance of the modified algorithms relative to the original algorithm.  They measure how likely  and by how much the modified algorithm  obtains better results  than the original algorithm.

\subsection{Test functions and details of experiments}

We use a total of 30 test functions of varying dimensions  listed in \cite[Table 2]{WNA2015} for our experiments. The list of test functions was compiled from various sources by the authors in \cite{WNA2015}. They include commonly used  test functions such as Ackley, Griewank, Powell, Rastrigin, Sphere etc; and are  of different kinds,  unimodal, multi-modal, separable and inseparable.
Majority of the functions in their table can be defined for arbitrary  dimension $d$. For such functions, we  consider  $d =5, 10, 20$ and $40$ to investigate the effect of dimension on the optimization algorithms. The other test functions (except one which can only be defined for dimension 4, which we exclude in our experiments) can only be defined for $d=2$.
In total, we have a collection of 70 test functions, see Appendix A, denoted by $\mathcal{C}$. We shall partition  $\mathcal{C}$ into different collections   according to $d=2, 5, 10, 20$ and $40$ in reporting our results.

For each $f \in \mathcal{C}$,  we conduct 100 runs of  algorithms $A$, $mA$ and $hmA$ for 10,000 iterations where $A$ = PSO, BAT, or CSO. To gauge the progress of each  run of an algorithm, we output the algorithm best functional value found  at $t$-th iteration, where $t=50, 100, 200, 400, 1000, 3000$ and $10000$.

Throughout the experiments, the number of agents/particles used is  $n=32$. For $mA$ or $hmA$ ($A =$ PSO, BAT or CSO), the stochastic perturbation used is  independent normal distributions  with mean 0 and standard deviation $ 0.005$.

We follow common parameter settings for PSO, BAT and CSO:
\begin{itemize}
\item PSO: $w=0.729$ (inertia weight), $c_1=c_2=1.5$ (acceleration constants);
\item BAT:  $Q_{\min}=0$ (minimum frequency), $Q_{\max}=100$ (maximum frequency), $r_0=0.5$ (pulse rate), $r_A=0.5$ (loudness);
\item CSO: $\phi =0$.
\end{itemize}

\subsection{Comparison  and results}

This subsection compares the performance of an algorithm $A$ versus one of its modification $B$ using several test functions taken from the set $\mathcal{F}$. The functions in the set are commonly used to compare performance of an algorithm in the engineering literature.  To study  how the performance on a problem depends on the dimension of the problem, each $\mathcal{F}$ here is the subset of $\mathcal{C}$ that contains functions of the same dimension $d=2,5,10,20$ or $40$.
 We are interested in the following two questions: How likely does $B$ outperform $A$?  How much does $B$ exceed the best possible value   from using  $A$ and $B$? Here $B=mA$ or $hmA$.
 \medskip

\noindent (a) {\it How likely will  $B$ outperform $A$?} \medskip

Let $A_f(t)$ and $B_f(t)$ denote the (theoretical) best functional values from algorithms $A$ and $B$ respectively for the test function $f$ up to the $t$-th iteration. Similarly, let $A_{r,f}(t)$ and $B_{r,f}(t)$ denote the best functional values output by algorithms $A$ and $B$ respectively  up to the $t$-th iteration at the $r$-th run. The probability that algorithm $B$ outperforms algorithm  $A$ for $f$ at the $t$-th iteration is $\Prob(B_f(t) < A_f(t))$, and it can be estimated by  $\sum_{r=1}^{100} I(B_{r,f}(t) < A_{r,f}(t))/100$. The average of   these estimates
over  $\mathcal{F}$, denoted by $P_{B \succ A}(t)$, can be thought of  as an average winning proportion of $B$ over $A$ (or simply, it's  winning proportion). Specifically,
$$ P_{B \succ A}(t):=\frac1{100 |\mathcal{F}|} \sum_{f \in \mathcal{F}} \sum_{r=1}^{100} I( B_{r,f}(t) < A_{r,f}(t)). $$
Similarly, interchanging $A$ and $B$ above, we define the winning proportion of $A$ over $B$, $ P_{A \succ B}(t)$.
Since $ P_{A \succ B}(t) +  P_{B \succ A}(t) =1$, it suffices to track $ P_{B \succ A}(t)$. Loosely speaking, if $ P_{B \succ A}(t) >1/2$,   $B$ is more likely to outperform $A$ for $f \in \mathcal{F}$. Indeed, the larger the  $ P_{B \succ A}(t)$, the more likely  $B$ outperforms  $A$.

\begin{figure}[t]
  \includegraphics[width=0.5\textwidth, height=0.25\textheight]{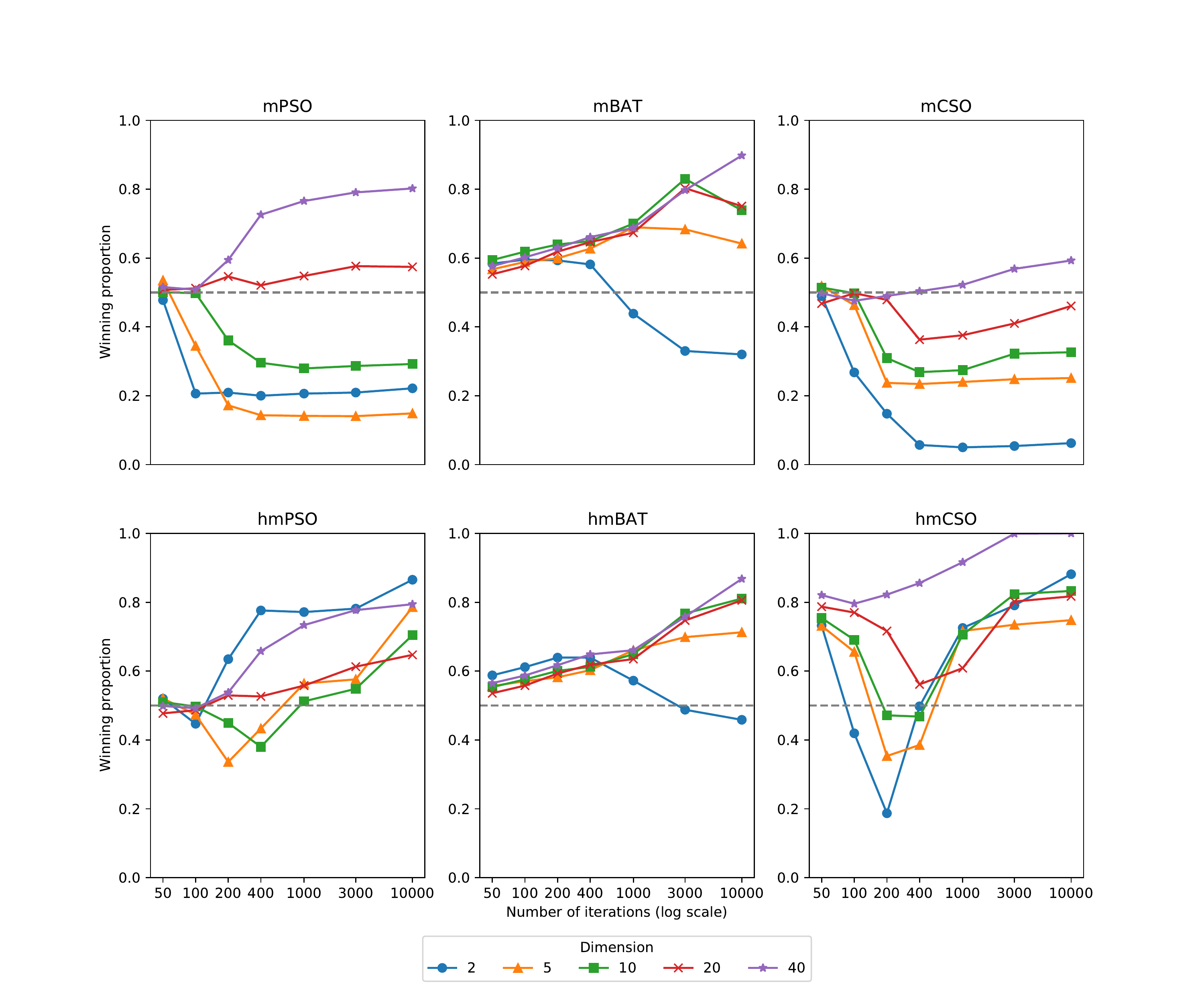}
  \caption{Plots of winning proportion of  $mA$ against $A$ (top row), and that of $hmA$ against $A$ (bottom row) where $A = $ PSO, BAT and CSO categorized according to the dimension of the test functions.}
  \label{fig:winning proportions}
\end{figure}

The numerical results are plotted in Figure \ref{fig:winning proportions}. Based on these plots, we observe have the followings.  When $t$ is fixed, the winning proportion  of PP/HPP modified algorithm outperforming  its counterpart  increases as the dimension of the test function increases. Both mBAT and hmBAT are comparable in their performance against BAT, and the winning proportion increases steadily from about 0.5 to 0.7 for dimension  at least 5. Interestingly, the HPP strategy has more significant enhancement effect on PSO and CSO. For test functions with dimension 40, winning proportion of hmCSO is nearly 1 after 3000 iterations.
\medskip

\noindent (b) {\it How much does $B$ (or $A$) outperform  $A$ (or  $B$)?} \medskip

For simplicity, we suppress the dependence of a notation   on other parameters/notations when the context is clear.
To this end, let  $m_*$ and  $m^*$   denote respectively,
the minimum (i.e., the best) and maximum (i.e., the worst) of the outputs  from 100 runs of $A$ and 100 runs of $B$ for $f$ at the $t$-th iteration.
We view $m_*$ as a proxy of the best possible output by $A$ and $B$; and  let $R = m^*-m_*$ denote the range of the  200 outputs of $A$ and $B$. Define the relative error  of $A$ with respect to $A$ and $B$, $RE_{A, A\cdot B, f}(t)$, to be the average of $(A_{r,f}(t) - m_*)/R$ over 100 runs; and the relative error of $B$ relative to $A$ and $B$, $RE_{B, A \cdot B, f}(t)$ is similarly defined. We  consider  the relative instead of  absolute  error in order  to cancel out the effect due to a scale change of $f$ (i.e., if we consider $cf$ instead of $f$). For brevity, we simply call this quantity the relative error.
Note that the relative errors lie between 0 and 1. Small relative error of $A$ (i.e., $ RE_{A, A\cdot B, f}(t)$)  indicates the results from algorithm $A$, as a whole,   are closer to the best possible value $m_*$. Further, if the relative error  of $A$ is less than that of $B$, then generally results from algorithm $A$ are closer to $m_*$ than those  by $B$'s.
A similar interpretation can be extended over a class of functions $f \in \mathcal{C}$ if we define   $RE(A, A\cdot B)(t):=\frac{1}{|\mathcal{F}|} \sum_{f \in \mathcal{F}} RE_{A, A \cdot B, f}(t)$ and call this the overshoot of $A$ relative to $A$ and $B$.

\begin{figure}[t]
\includegraphics[width=0.5\textwidth, height=0.25\textheight]{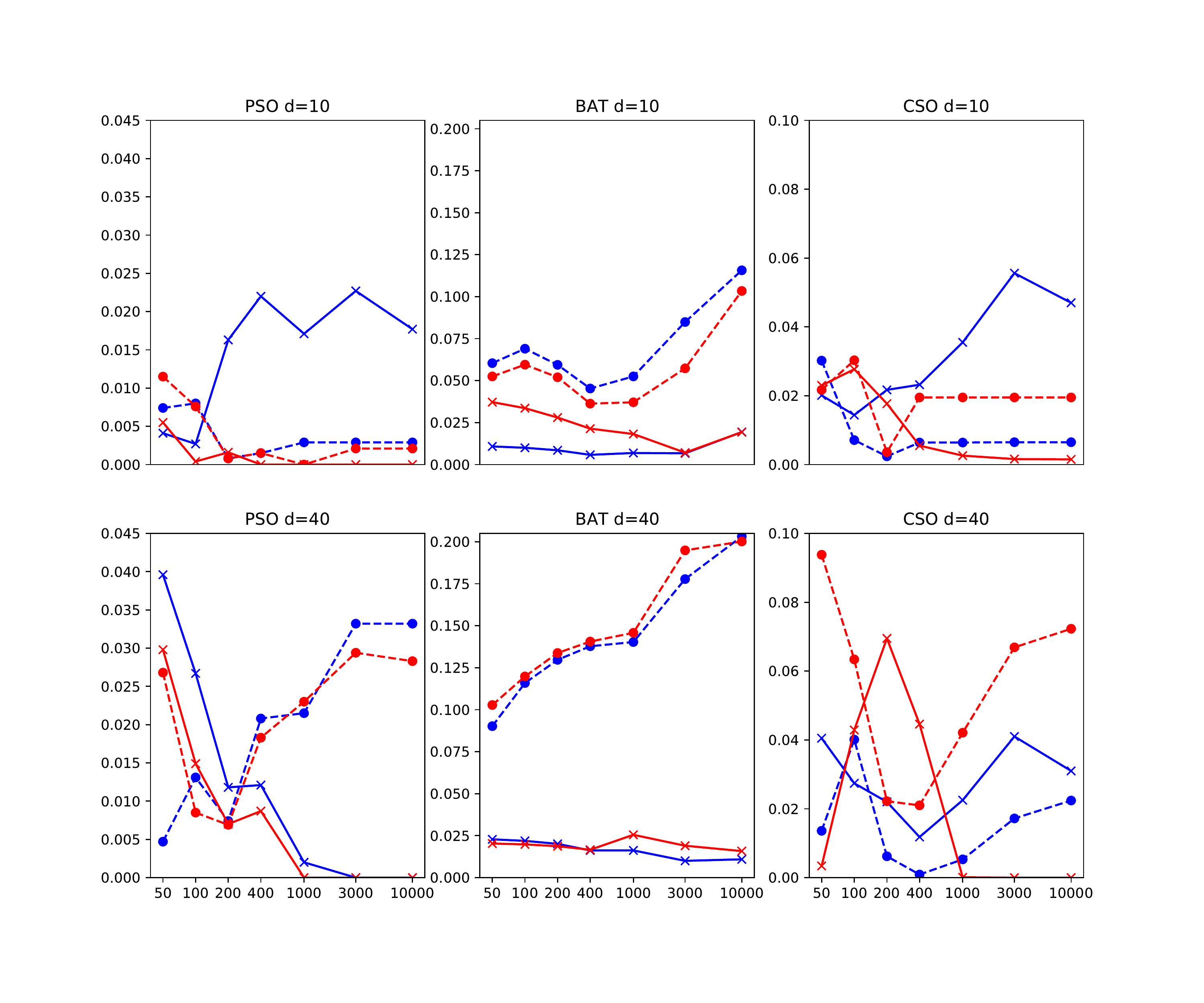}
  \caption{Plots of relative error  of $A$ relative to $A$ and $mA$ (dashed blue curve), and relative error  of $mA$ relative to $A$ and $mA$ (solid blue curves). Plots of relative error of $A$ relative to $A$ and $hmA$ (dashed red curves), and relative error of $hmA$ relative to $A$ and $hmA$ (solid red curves).  }
  \label{fig:winning margins}
\end{figure}

Figure \ref{fig:winning margins} plots the relative errors of $A$  (dashed blue  curve) and  $mA$ (solid blue  curve) with respect to one another; and the relative errors of $A$  (dashed red curve) and  $hmA$ (solid red  curve) with respect to one another. We only display the plots for $d=10$ and $d=40$. We highlight some observations from these plots in Figure \ref{fig:winning margins}.  (i) hmBAT and mBAT outperform BAT by a significant margin. (ii) At $d=40$, both mPSO and hmPSO are nearly 0 from 1000 iterations onwards. This implies mPSO and hmPSO  significantly outperform the basic  PSO from 1000 iterations. At $d=10$, only hmPSO (resp.,  hmCSO) improves the basic PSO (resp., CSO) from 400 iterations.

Combining the observations in (a) and (b), Figures \ref{fig:winning proportions} and \ref{fig:winning margins} convincingly demonstrate that the HPP strategy significantly improves BAT, CSO and PSO particularly when the dimension of a test function is large. Improvement is measured in two complementary aspects,   ``how likely'' and ``how much''  one algorithm outperforms the other.

\subsection{Tuning parameter analysis for PP and HPP strategies }

As demonstrated in Section V-B,  additional stochastic component in a modified swarm-based algorithm enhances the exploration ability of the algorithm under consideration. The stochastic component chosen is  Gaussian  with mean 0 and standard deviation, $\sigma=0.005$.  Expectedly the noise level, in terms of  $\sigma$ or a heavier-tailed distribution, affects the efficiency of the modified algorithm since  strong level of noise interferes  the algorithm's  exploitation ability.  We conduct further numerical experiments with the same set-up as above to  investigate
the effects of $\sigma$ in the Gaussian noise and $t$-distributions, which have heavier tails than the Gaussian, on the performance of $mA$ and $hmA$.
\vspace{1em}

\noindent (i) {\it  On the standard deviation of the Gaussian noise}~ We tried $\sigma =0.005, 0.01, 0.02$ and $0.05$. The plots  of $P_{B\prec A}(t)$ (analogous to Figure \ref{fig:winning proportions}) and $RE(A, A\cdot B)(t)$ (analogous to Figure \ref{fig:winning margins})  are very similar for $hmA$ ($A=$ PSO, BAT and CSO) for all the $\sigma$s; and $\sigma =0.005$ and $0.01$. Hence, we recommend using $\sigma$ in $[0.005, 0.01]$.  Interested readers can refer to Figures \ref{fig:proportion_s2} and \ref{fig:proportion_s3} in Appendix \ref{section:additional numerical results}.
\vspace{1em}

\noindent (ii) {\it  On the choice of distribution} ~ We examine the effect on the performance of the modified algorithms if we change the  Gaussian distribution with a heavier-tailed distribution. We choose distributions $0.01 \sqrt{(m-2)/m} \ t_m $ for $m = 5, 10, 30$ and $60$ where $t_m$ denotes a $t$-distribution with $m$ degrees of freedom. The scaling factor $0.01 \sqrt{(m-2)/m}$ is to ensure the scaled $t$-distribution has the same standard deviation as that of the Gaussian noise in (i) above. Note also that $t_m$ converges to a standard normal. For $m$ large, we do not expect  essential difference between the $t$-distribution and the Gaussian noise.

The analogous plots of  $P_{B\prec A}(t)$ and $RE(A, A\cdot B)(t)$ using $t$-distributions for the choices of degrees of freedom (df)   look almost the same as the corresponding plots using Gaussian noise. Interested readers can view these plots as Figure \ref{fig:proportion_s4}-\ref{fig:proportion_s5} in Appendix \ref{section:additional numerical results}.

\section{Conclusion}

In this article, we delineate sufficient conditions for a swarm-based algorithm to converge almost surely to  an optimum of an objective  function.  This assumes that there is no constraint on the computation budget and the algorithm is allowed to run indefinitely. The class of objective functions for which Theorem \ref{ascon} holds is very large. We proposed two strategies, perturbation-project (PP) and heterogeneous perturbation-projection (HPP), to modify a given swarm-based algorithm so that the sufficient conditions (C1)--(C3) hold, and hence Theorem \ref{ascon} guarantees the modified algorithm converges almost surely to the global optimum. The proposed strategies are simple to implement and not algorithm-specific. We demonstrated how these strategies are applied to  PSO, Bat Algorithm and CSO. Extensive numerical experiments were conducted and the results  show that the modified algorithm either out-performs or performs on par with the original algorithm over a finite computational budget, and especially so for HPP modified algorithm.   Our simulation results suggest that  applying the HPP strategy with mean 0 and standard deviation 0.005 Gaussian distribution for the additional stochastic component is likely to be effective in improving a swarm-based algorithm.

To conclude this article, we report two interesting observations from our experiments. The first observation is that
  when $A$ outperforms $mA$ or $hmA$,  $A$ frequently produces a marginally better result; however,  when $mA$ or $hmA$ outperforms $A$, the modified version quite often outperforms $A$ by  relatively larger margin.
  We offer a heuristic argument. When both $A$ and $mA$ are exploring  in the same neighbourhood  of a local minimum, the stochastic element hinders the exploitation by a small amount and hence $A$ is more likely to produce only marginally better solution. However, when the occasion for $mA$ to make a big jump leaving the local minimum's neighbourhood arises,  $mA$ has the potential to   explore a better region and hence produces noticeably better solution than the original algorithm does.
  The same heuristic reasoning carries over to the comparison of $A$ and $hmA$. Moreover, $hmA$ version due to the built in diversity of agents,  better balances the conflicting demands of exploitation and exploration, and expectedly amplifies this  observation.

The second observation comes from our application of PP and HPP strategies to Differential Evolution (DE), an evolutionary algorithm.  The same numerical experiments were conducted to compare DE with mDE, and DE with hmDE. We found that DE outperforms mDE and hmDE; the plots  can be found in Appendix \ref{section:differential evolution} as Figures \ref{fig:proportion_s6}-\ref{fig:margin_s7}.  A reason may be that DE already performs very well in the optimization problem over majority of the test functions in the experiment and so any additional stochastic component will either not help or even worsen the search solution.  Moreover, the random perturbation we apply here mimic the random local exploration of an ant or other animal. Genetic mutation in general cannot describe such dynamics and require some other more structured mechanism.


%

\appendices

\section*{Acknowledgment}
Choi's research was supported by the Singapore MOE Academic Research Funds R-155-000-222-114. Lai's research was supported by the National Science Foundation under DMS-1811818.   Tong's research was supported by MOE Academic Research Funds R-146-000-292-114. Wong's research was supported by the National Institutes of Health under R01GM107639.





%
\bibliography{ref}

\newpage

\renewcommand{\thetable}{S\arabic{table}}
\renewcommand{\thefigure}{S\arabic{figure}}

\section{Test functions used in the numerical simulations}
\begin{longtable}{|ccccc|}

    \hline
    Label & Name &  min f(x) & Unimodal & Separable \\
    \specialrule{1.5pt}{1pt}{1pt}
    \endfirsthead


    \hline
    \endfoot

    \endlastfoot

    F1               & Ackley               &  0             & No        & No         \\ \hline
    \multicolumn{5}{|l|}{$ -20\exp(-0.2\sqrt{\frac{1}{n}\sum_{i=1}^{n}x_i^2})-\exp(\frac{1}{n}\sum_{i=1}^{n}\cos(2\pi x_i))+ 20 + e$} \\ \specialrule{1.5pt}{1pt}{1pt}
    F2               & Bohachevsky2         &  0             & No        & No         \\ \hline
    \multicolumn{5}{|l|}{$ \sum_{i=1}^{d-1} \left[ x_i^2 + 2x_{i+1}^2 -0.3\cos(3\pi x_i)\cos(4\pi x_{i+1})+0.3 \right]$} \\ \specialrule{1.5pt}{1pt}{1pt}
    F3               & Bohachevsky3         &  0             & No        & No         \\ \hline
    \multicolumn{5}{|l|}{$ \sum_{i=1}^{d-1} \left[ x_{i+1}^2 + 2x_{i+2}^2 -0.3\cos(3\pi x_i+4\pi x_{i+1})+0.3 \right]$} \\ \specialrule{1.5pt}{1pt}{1pt}
    F4               & Bukin6               &  0             & No        & No         \\ \hline
    \multicolumn{5}{|l|}{$100\sqrt{|x_2 - 0.01x_1^2|} + 0.01|x_1+10|$} \\ \specialrule{1.5pt}{1pt}{1pt}
    F5               & DropWave             &  -1             & No        & No         \\ \hline
    \multicolumn{5}{|l|}{$  -\frac{1+\cos(12\sqrt{x_1^2+x_2^2})}{0.5(x_1^2+x_2^2) + 2}$)} \\ \specialrule{1.5pt}{1pt}{1pt}
    F6               & Eggholder     &  -959.6407     & No        & No         \\ \hline
    \multicolumn{5}{|l|}{$  -(x_2 + 47) \sin(\sqrt{|x_2 + \frac{x_1}{2} +47|}) - x_1\sin(\sqrt{|x_1-(x_2+47)|})$)} \\ \specialrule{1.5pt}{1pt}{1pt}
    F7               & GoldSteinPrice        &  3             & No        & No         \\ \hline
    \multicolumn{5}{|l|}{$[1 + (x_1 + x_2 + 1)^2(19 - 14x_1+3x_1^2- 14x_2 + 6x_1x_2 + 3x_2^2)] \times$} \\ 
    \multicolumn{5}{|l|}{$[30 + (2x_1 - 3x_2)^2(18 - 32x_1 + 12x_1^2 + 4x_2 - 36x_1x_2 + 27x_2^2)]$} \\ \specialrule{1.5pt}{1pt}{1pt}
    F8               & Griewank              &  0             & No        & No         \\ \hline
    \multicolumn{5}{|l|}{$1 + \sum_{i=1}^d \frac{x_i^2}{4000} - \prod_{i=1}^d \cos(\frac{x_i}{\sqrt{i}})$} \\ \specialrule{1.5pt}{1pt}{1pt}
    F9               & McCormick      &  -1.9133       & No        & No         \\ \hline
    \multicolumn{5}{|l|}{$  \sin(x_1+x_2) + (x_1-x_2)^2 - 1.5x_1 + 2.5x_2 + 1$} \\ \specialrule{1.5pt}{1pt}{1pt}
    F10              & Schaffer2            &  0             & No        & No         \\ \hline
    \multicolumn{5}{|l|}{$  0.5 + \frac{\sin^2(x_1^2 - x_2^2) - 0.5}{(1 + 0.001(x_1^2 + x_2^2))^2}$} \\ \specialrule{1.5pt}{1pt}{1pt}
    F11              & Schaffer4            &  0.292579             & No        & No         \\ \hline
    \multicolumn{5}{|l|}{$  0.5 + \frac{\cos^2(\sin(\lvert x_1^2 - x_2^2 \rvert)) - 0.5}{(1 + 0.001(x_1^2 + x_2^2))^2}$)} \\ \specialrule{1.5pt}{1pt}{1pt}
    F12              & Bohachevsky1          &  0             & No        & Yes        \\ \hline
    \multicolumn{5}{|l|}{$ \sum_{i=1}^{d-1} x_i^2 + 2x_{i+1}^2 -0.3\cos(3\pi x_i)-0.4\cos(4\pi x_{i+1})+0.7$} \\ \specialrule{1.5pt}{1pt}{1pt}
    F13              & Booth                &  0             & No        & Yes        \\ \hline
    \multicolumn{5}{|l|}{$(x_1+2x_2-7)^2+(2x_1+x_2-5)^2$} \\ \specialrule{1.5pt}{1pt}{1pt}
    F14              & Branin        &  0.397887      & No        & Yes        \\ \hline
    \multicolumn{5}{|l|}{$  (x_2 - 
    \frac{5.1}{4\pi^2}x_1^2 + \frac{5}{\pi}x_1^2 - 6)^2 + 10(1-\frac{1}{8\pi}\cos x_1 + 10$)} \\ \specialrule{1.5pt}{1pt}{1pt}
    F15              & Michalewicz5 &  -4.687658     & No        & Yes        \\ \hline
    \multicolumn{5}{|l|}{$ -\sum_{i=1}^5 \sin(x_i) \sin^{20} (\frac{ix_i^2}{\pi})$} \\ \specialrule{1.5pt}{1pt}{1pt}
    F16              & Rastrigin           &  0             & No        & Yes        \\ \hline
    \multicolumn{5}{|l|}{$ 10d + \sum_{i=1}^d [x_i^2 - 10\cos(2\pi x_i)]$} \\ \specialrule{1.5pt}{1pt}{1pt}
    F17              & Shubert        &  -186.73       & No        & Yes        \\ \hline
    \multicolumn{5}{|l|}{$ \prod_{j=1}^2 (\sum_{i=1}^5 i \cos((i+1)x_j + i))$} \\ \specialrule{1.5pt}{1pt}{1pt}
    F18              & Beale               &  0             & Yes       & No         \\ \hline
    \multicolumn{5}{|l|}{$  (1.5-x_1+x_1x_2)^2+(2.25-x_1+x_1x_2^2)^2+(2.625-x_1+x_1x_2^3)^2$} \\ \specialrule{1.5pt}{1pt}{1pt}
    F19              & DixonPrice            &  0             & Yes       & No         \\ \hline
    \multicolumn{5}{|l|}{$ (x_1-1)^2 + \sum_{i=2}^d i(2x_i^2-x_{i-1})^2$)} \\ \specialrule{1.5pt}{1pt}{1pt}
    F20              & Easom                &  -1            & Yes       & No         \\ \hline
    \multicolumn{5}{|l|}{$  -\cos(x_1)\cos(x_2)\exp (-(x_1-\pi)^2-(x_2-\pi)^2)$)} \\ \specialrule{1.5pt}{1pt}{1pt}
    F21              & Matyas                &  0             & Yes       & No         \\ \hline
    \multicolumn{5}{|l|}{$  0.26(x_1^2 + x_2^2) - 0.48x_1x_2$} \\ \specialrule{1.5pt}{1pt}{1pt}
    F22              & Powell                &  0             & Yes       & No         \\ \hline
    \multicolumn{5}{|l|}{$ \sum_{i=1}^{d/4} [(x_{4i-3} + 10x_{4i-2})^2 + 5(x_{4i-1} - x_{4i})^2 + (x_{4i-2} - 2x_{4i-1})^4 + 10(x_{4i-3} - x_{4i})^4]$} \\ \specialrule{1.5pt}{1pt}{1pt}
    F23              & Rosenbrock           &  0             & Yes       & No         \\ \hline
    \multicolumn{5}{|l|}{$ \sum_{i=1}^{d-1} [100(x_{i+1} - x_i^2)^2 + (x_i - 1)^2]$} \\ \specialrule{1.5pt}{1pt}{1pt}
    F24              & Schwefel &  -418.9829d    & Yes       & No         \\ \hline
    \multicolumn{5}{|l|}{$ 418.9829d - \sum_{i=1}^d x_i \sin(\sqrt{|x_i|})$} \\ \specialrule{1.5pt}{1pt}{1pt}
    F25              & Trid6 &  -d(d+4)(d-1)/6 & Yes       & No         \\ \hline
    \multicolumn{5}{|l|}{$ \sum_{i=1}^d (x_i - 1)^2 - \sum_{i=1}^d x_ix_{i-1}$} \\ \specialrule{1.5pt}{1pt}{1pt}
    F26              & Zakharov              &  0             & Yes       & No         \\ \hline
    \multicolumn{5}{|l|}{$ \sum_{i=1}^d x_i^2 + (\sum_{i=1}^d 0.5ix_i)^2 + (\sum_{i=1}^d 0.5ix_i)^4$} \\ \specialrule{1.5pt}{1pt}{1pt}
    F27              & Sphere              &  0             & Yes       & Yes        \\ \hline
    \multicolumn{5}{|l|}{$ \sum_{i=1}^d x_i^2$} \\ \specialrule{1.5pt}{1pt}{1pt}
    F28              & Sumsquare            &  0             & Yes       & Yes        \\ \hline
    \multicolumn{5}{|l|}{$ \sum_{i=1}^d ix_i^2$} \\ \hline

    \captionsetup{skip=30pt}
    \caption{Table of test functions containing their function label, name, known minimum value, unimodality and separability.
    Functions F4, 5, 6, 7, 9, 10, 11, 13, 14, 17, 18, 20, 21 have dimension 2, and F15 is dimension 5. All other functions were
    tested at dimensions 5, 10, 20 and 40.}
    \label{tab:funcinfo}

\end{longtable}

\newpage
\section{Additional numerical results}
\label{section:additional numerical results}
Here we report the simulation results of  dimensions (Figure \ref{fig:margin_s1}) and noise setups (Figures \ref{fig:proportion_s2}-\ref{fig:proportion_s5}).
\begin{figure}[ht]
    \centerline{\includegraphics[width=\columnwidth, height=0.8\textheight]{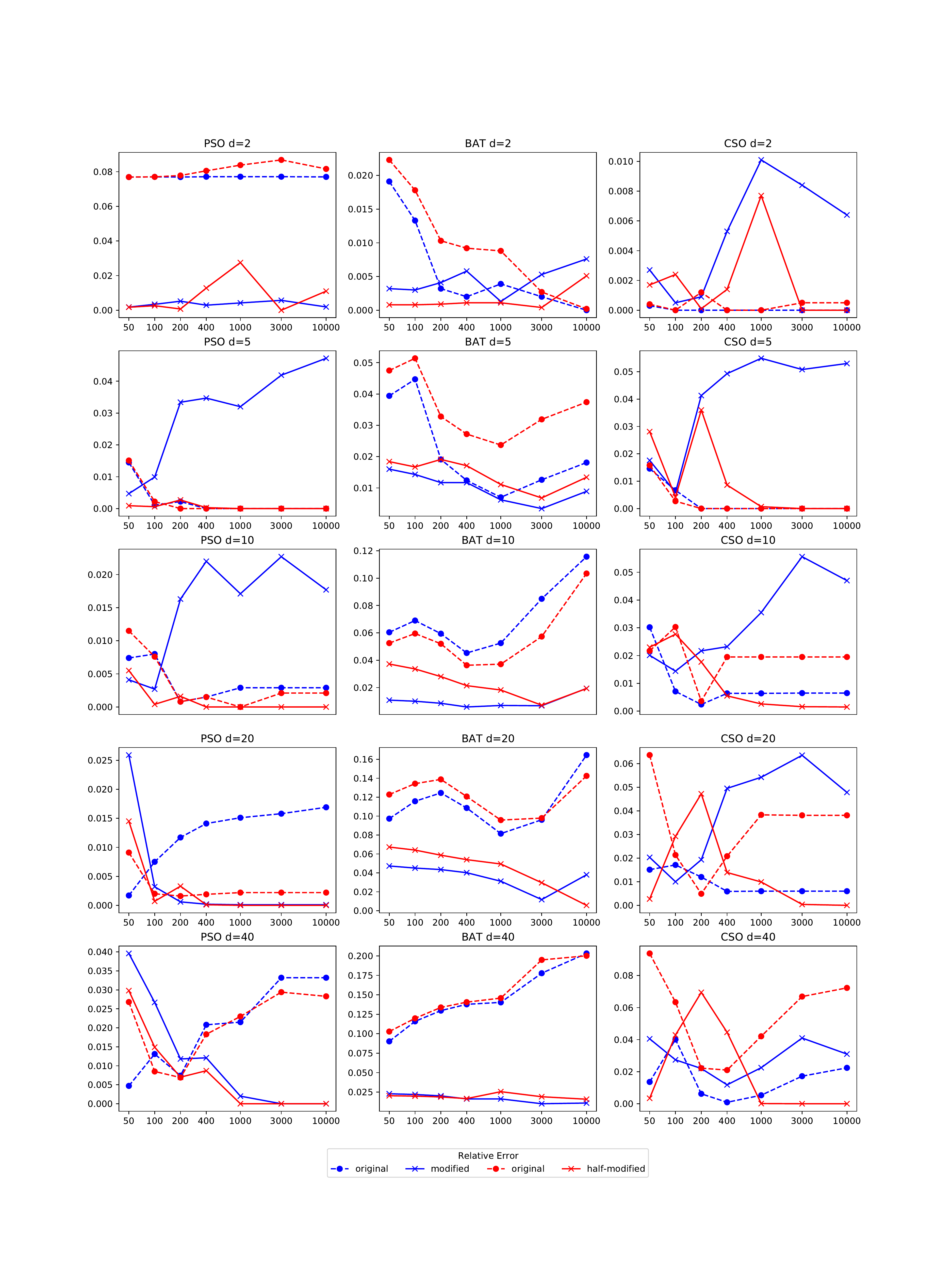}}
    \caption{Plots of relative error  of $A$ relative to $A$ and $mA$ (dashed blue curve), and relative error
    of $mA$ relative to $A$ and $mA$ (solid blue curves). Plots of relative error of $A$ relative to $A$ and $hmA$
    (dashed red curves), and relative error of $hmA$ relative to $A$ and $hmA$ (solid red curves).  }
    \label{fig:margin_s1}
\end{figure}

\begin{figure}[ht]
    \centerline{\includegraphics[width=\columnwidth, height=0.8\textheight]{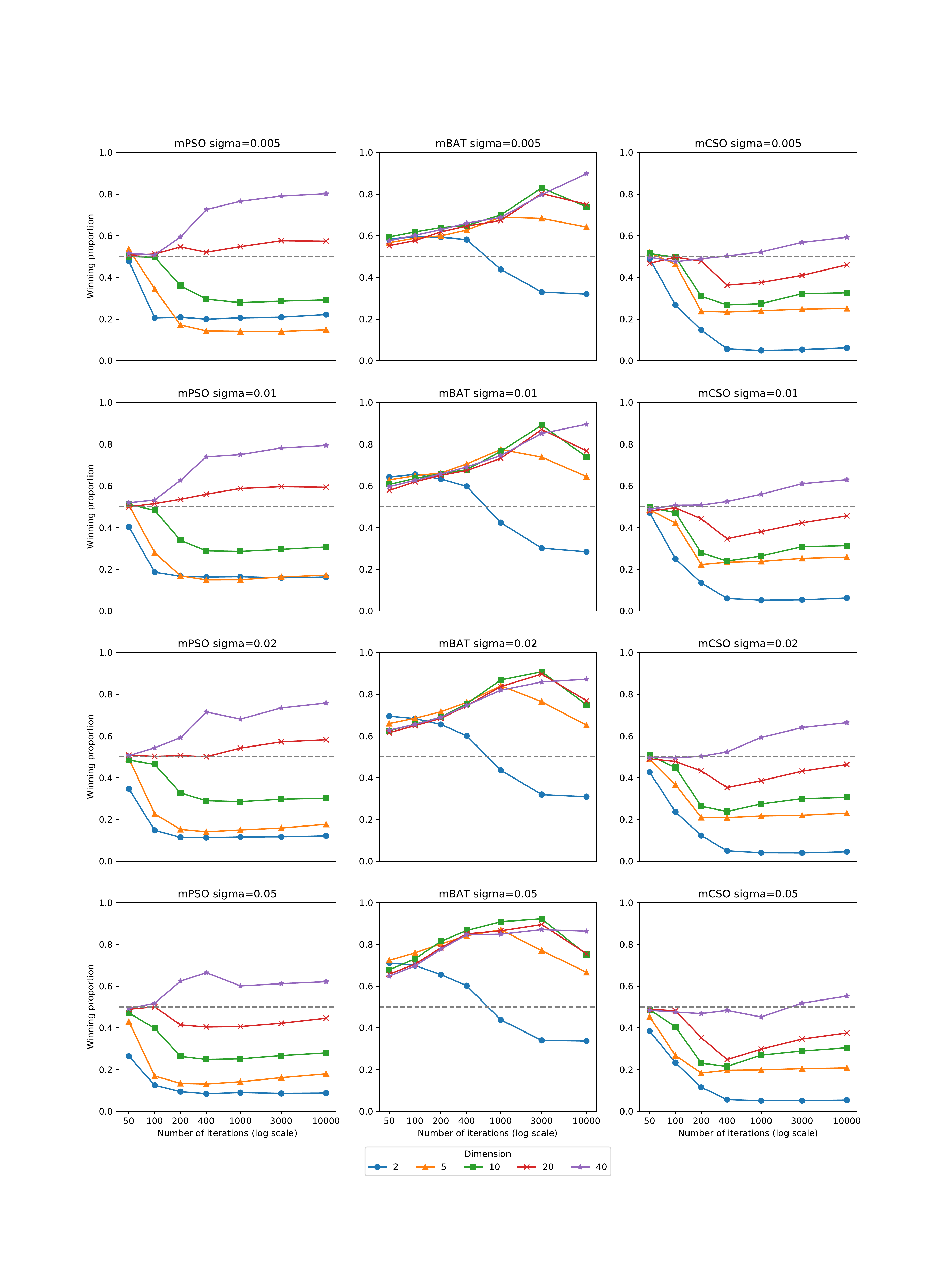}}
    \caption{Plots of winning proportion of $mA$ against $A$ where $A = $ PSO, BAT and CSO categorized
    according to the dimension of the test functions. Each row shows the result for one value of sigma $\sigma$,
    where $\sigma = 0.005, 0.01, 0.02, 0.05$.}
    \label{fig:proportion_s2}
\end{figure}

\begin{figure}[ht]
    \centerline{\includegraphics[width=\columnwidth, height=0.8\textheight]{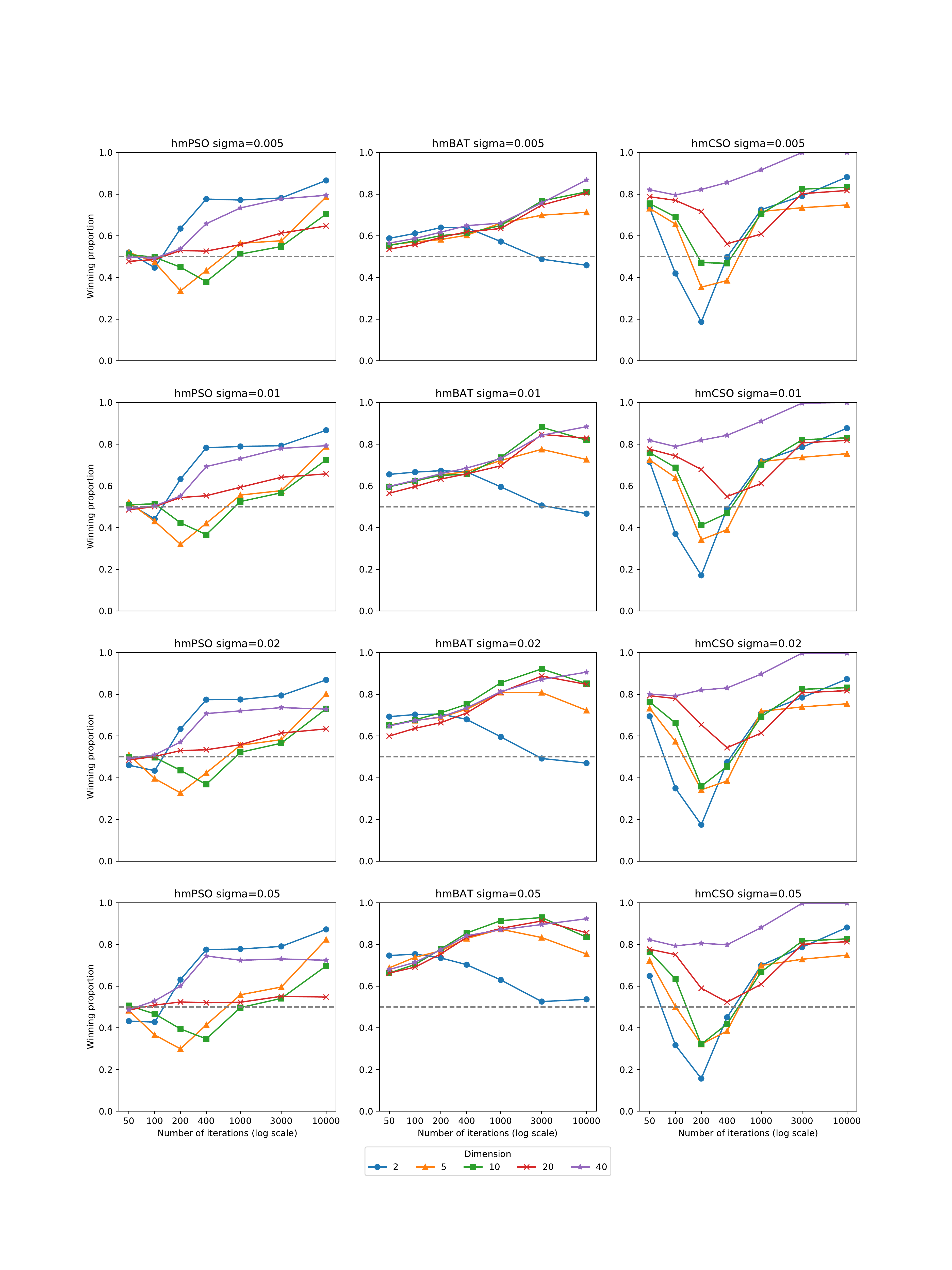}}
    \caption{Plots of winning proportion of $hmA$ against $A$ where $A = $ PSO, BAT and CSO categorized
    according to the dimension of the test functions. Each row shows the result for one value of sigma $\sigma$,
    where $\sigma = 0.005, 0.01, 0.02, 0.05$.}
    \label{fig:proportion_s3}
\end{figure}

\begin{figure}[ht]
    \centerline{\includegraphics[width=\columnwidth, height=0.8\textheight]{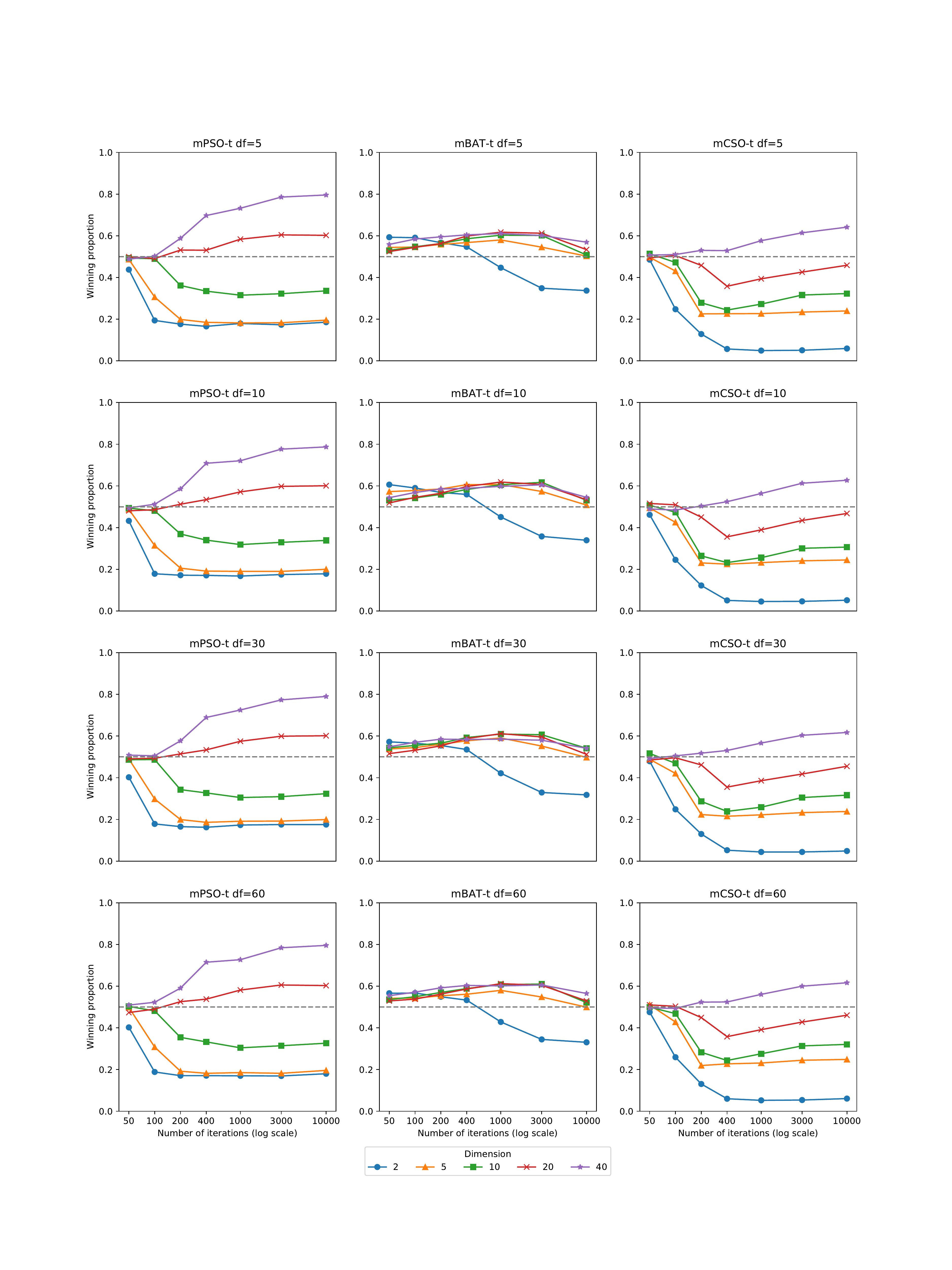}}
    \caption{Plots of winning proportion of $mA-t$ against $A$ where $A = $ PSO, BAT and CSO categorized
    according to the dimension of the test functions. Each row shows the result for one value of degrees
    of freedom $df$, where $df = 0.005, 0.01, 0.02, 0.05$.}
    \label{fig:proportion_s4}
\end{figure}

\begin{figure}[ht]
    \centerline{\includegraphics[width=\columnwidth, height=0.8\textheight]{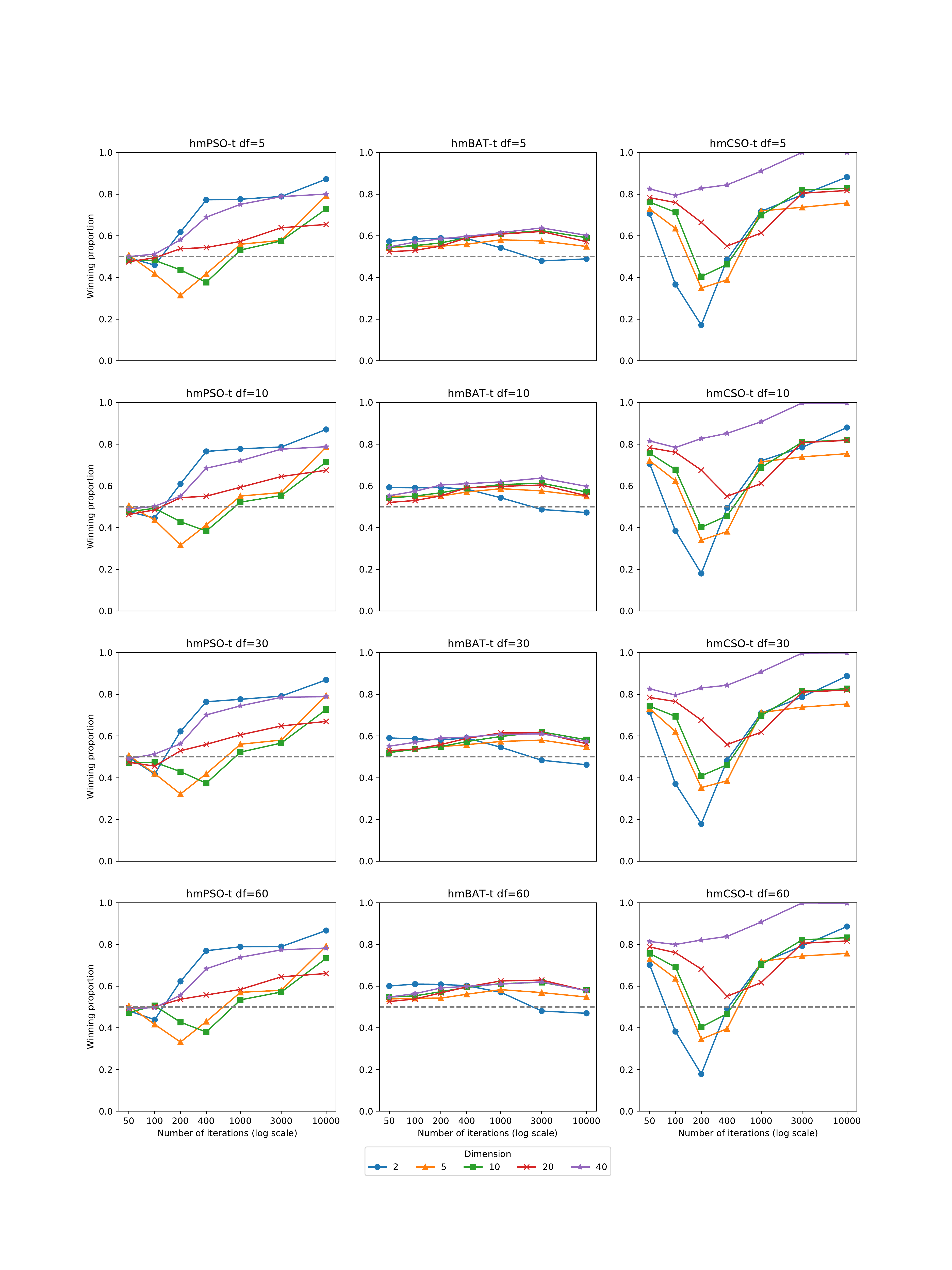}}
    \caption{Plots of winning proportion of $hmA-t$ against $A$ where $A = $ PSO, BAT and CSO categorized
    according to the dimension of the test functions. Each row shows the result for one value of degrees
    of freedom $df$, where $df = 0.005, 0.01, 0.02, 0.05$.}
    \label{fig:proportion_s5}
\end{figure}

\pagebreak

\clearpage

\section{Differential Evolution (DE) and its modification}
\label{section:differential evolution}
Here we review the DE algorithm and call the particles in DE as  agents.  In each DE iteration, each agent considers a new proposed location and move to the new location if it provides a better function value.
The  proposed movement is obtained by truncating the difference between two other random agents at a random component. More specifically,  DE  repeats steps 1 to 6 until the user-specified termination criterion is met.
\begin{enumerate}
\item Repeat the following step for each agent $\bfx_i(t)$
\item Randomly find two other agents $\bfx_j(t)$ and $\bfx_k(t)$.
\item Compute $\bfy_i(t)=\bfx_i(t)+F(\bfx_j(t)-\bfx_k(t))$, where $F\in [0,2]$ is a constant known as \emph{differential weight} and its default value is $0.8$.
\item Randomly find a component index $\ell \in \{1,\ldots, d\}$.
\item For component $m \in \{1, \ldots , d\} \setminus \{\ell\}$, replace the $m$th component of $\bfy_i(t)$ by the corresponding component of $\bfx_i(t)$  with probability $1-\omega$, where  $\omega\in [0,1]$  is a constant known as \emph{crossover probability}, and its default value is 0.9.
\item If $f(\bfx_i(t))>f(\bfy_i(t))$, let $\bfx_{i}(t+1)=\bfy_{i}(t)$; otherwise let $\bfx_{i}(t+1)=\bfx_{i}(t)$.
\end{enumerate}

Our modified DE, denoted by mDE, is to add an additional  step for step 6\\
 \text{6'. Let }$\bfy'_{i}(t)=\chi_{\calD}(\chi_{\calD}(\bfy_i(t))+\bfw_{\ell}(t))$ and
\begin{equation}
\label{eqn:demod}
\bfx_{i}(t+1)=\begin{cases}
\bfy'_{i}(t),\quad \text{if } f(\bfy'_i(t))<f(\bfx_i(t));\\
\bfx_{i}(t),\quad \text{if } f(\bfy'_i(t))\geq f(\bfx_i(t)).\\
\end{cases}
 \end{equation}

The simulations use the same setups as the ones for PSO. The results are plotted in Figures \ref{fig:proportion_s6} and \ref{fig:margin_s7}. The modification actually worsen the optimization output.

\begin{figure}[ht]
    \centerline{\includegraphics[width=\columnwidth, height=0.4\textheight]{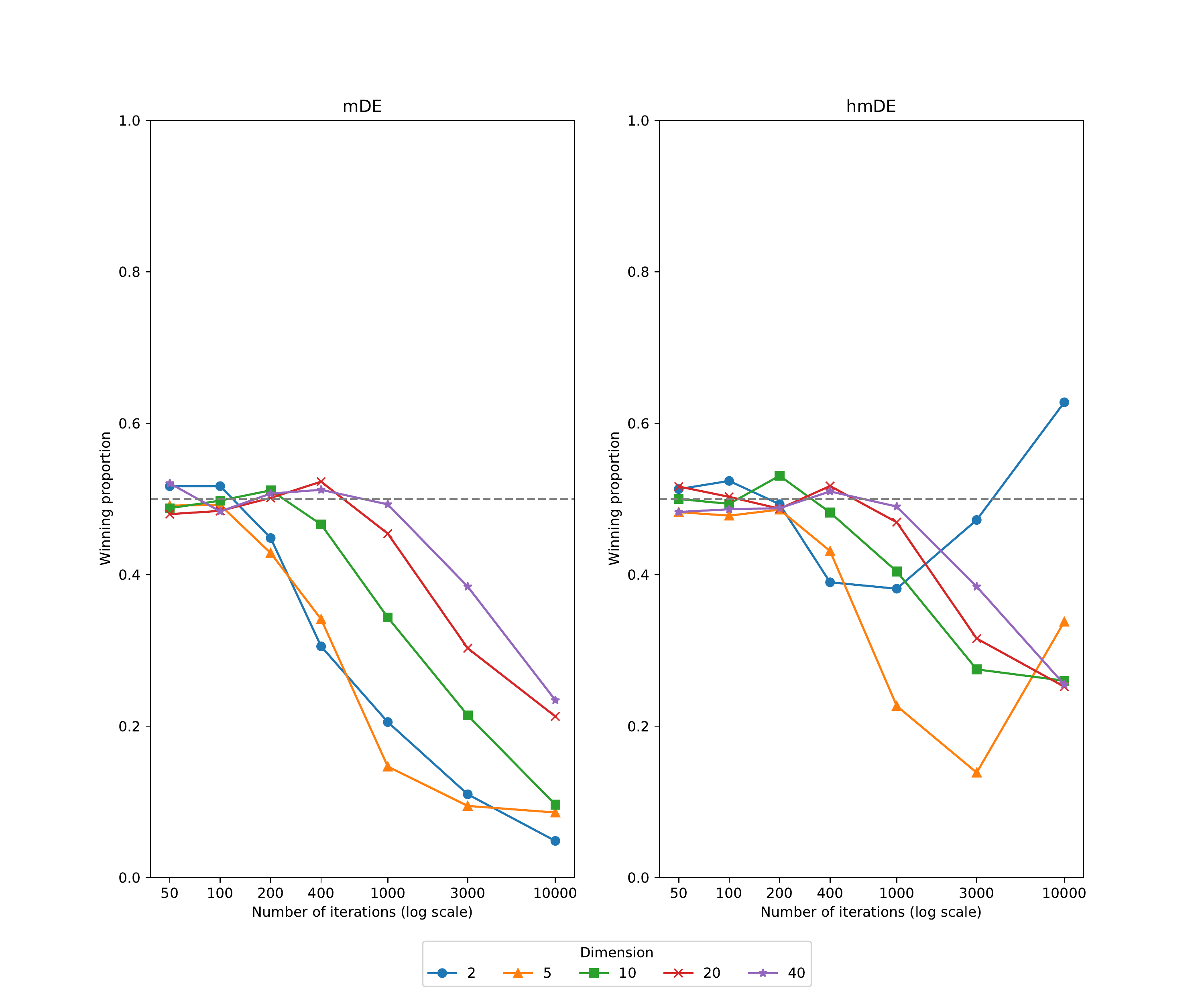}}
    \caption{Plots of winning proportion of $mDE$ against $DE$ (left), and that of $hmDE$ against $DE$ (right)
    where DE categorized according to the dimension of the test functions.}
    \label{fig:proportion_s6}
\end{figure}

\begin{figure}[ht]
    \centerline{\includegraphics[width=\columnwidth, height=0.4\textheight]{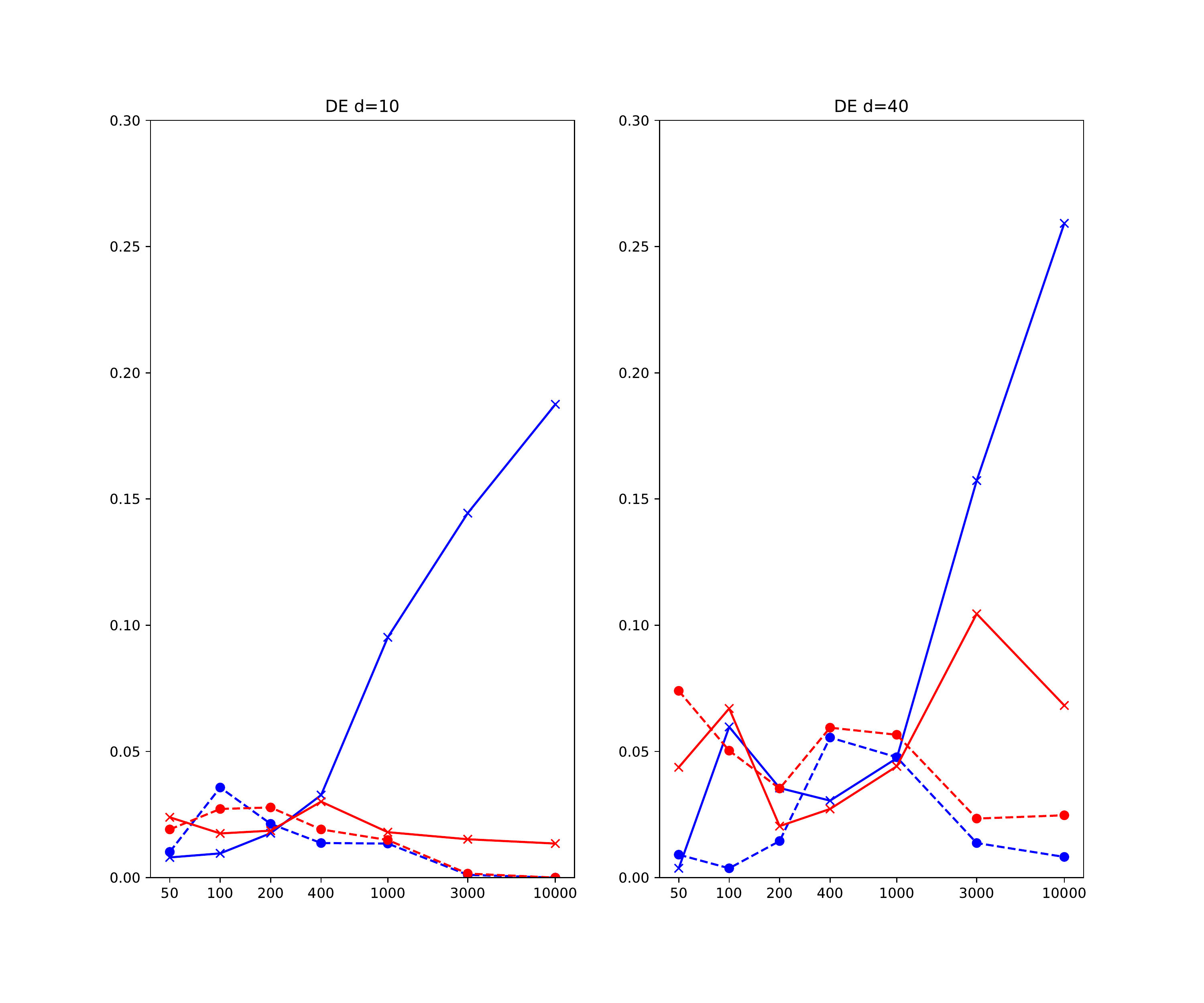}}
    \caption{Plots of relative error  of $DE$ relative to $DE$ and $mDE$ (dashed blue curve), and relative error
    of $mDE$ relative to $DE$ and $mDE$ (solid blue curve). Plots of relative error of $DE$ relative to $DE$ and $hmDE$
    (dashed red curves), and relative error of $hmDE$ relative to $DE$ and $hmDE$ (solid red curve).}
    \label{fig:margin_s7}
\end{figure}

\end{document}